\documentclass[12pt]{amsart}
\usepackage{enumerate}
\usepackage{hyperref}
\usepackage{graphicx,color}
\usepackage{cite}
\usepackage{amsthm, amsmath, amssymb, mathtools, braket, url}
\usepackage[top=45truemm, bottom=45truemm, left=30truemm, right=30truemm]{geometry}

\renewcommand{\epsilon}{\varepsilon}

\DeclareMathOperator{\Tr}{Tr}

\DeclareMathOperator{\sr}{s-rank}
\newcommand{\ketbra}[2]{\ket{#1}\!\bra{#2}}
\DeclareMathOperator*{\conv}{conv}
\DeclareMathOperator*{\capa}{capa}
\newcommand{\Sep}{\mathsf{Sep}}
\newcommand{\Herm}{\mathsf{Herm}}
\newcommand{\PSD}{\mathsf{PSD}}
\DeclareMathOperator*{\vspan}{span}
\DeclareMathOperator*{\vvec}{vec}

\newcounter{count}
\newcommand{\num}{\stepcounter{count}\the\value{count}}

\numberwithin{equation}{section}

\newtheorem{theorem}{Theorem}[section]
\newtheorem{lemma}[theorem]{Lemma}

\newtheorem{definition}[theorem]{Definition}
\newtheorem{proposition}[theorem]{Proposition}

\theoremstyle{definition}
\newtheorem{remark}[theorem]{Remark}
\newtheorem{example}[theorem]{Example}

\newcommand{\cF}{\mathcal{F}}

\newcommand{\cH}{\mathcal{H}}

\newcommand{\cK}{\mathcal{K}}
\newcommand{\cL}{\mathcal{L}}
\newcommand{\cM}{\mathcal{M}}

\newcommand{\cS}{\mathcal{S}}

\newcommand{\cV}{\mathcal{V}}

\title{Maximum Dimension of Subspaces with No Product Basis}

\author[Y.\ Yoshida]{Yuuya Yoshida}
\address{Yuuya Yoshida\\
Graduate School of Mathematics\\ Nagoya University\\ Furo-cho\\ Chikusa-ku\\ Nagoya\\ 464-8602\\ Japan}
\curraddr{}
\email{m17043e@math.nagoya-u.ac.jp}



\begin{document}
\maketitle

\begin{abstract}
Let $n\ge2$ and $d_1,\ldots,d_n\ge2$ be integers, and $\mathcal{F}$ be a field. A vector $u\in\mathcal{F}^{d_1}\otimes\cdots\otimes\mathcal{F}^{d_n}$ is called a product vector if $u=u^{[1]}\otimes\cdots\otimes u^{[n]}$ for some $u^{[1]}\in\mathcal{F}^{d_1},\ldots,u^{[n]}\in\mathcal{F}^{d_n}$. A basis composed of product vectors is called a product basis. In this paper, we show that the maximum dimension of subspaces of $\mathcal{F}^{d_1}\otimes\cdots\otimes\mathcal{F}^{d_n}$ with no product basis is equal to $d_1d_2\cdots d_n-2$ if either (i) $n=2$ or (ii) $n\ge3$ and $\#\mathcal{F}>\max\{d_i : i\not=n_1,n_2\}$ for some $n_1$ and $n_2$. When $\mathcal{F}=\mathbb{C}$, this result is related to the maximum number of simultaneously distinguishable states in general probabilistic theories (GPTs).
\end{abstract}

\section{Introduction}\label{intro}

Quantum theory is described by operators on complex Hilbert spaces, 
and realizes quantum information processing beyond classical information processing.
Although quantum theory and related topics have been studied so far, 
a mathematical foundation of quantum theory does not suffice.
For example, when Alice and Bob have quantum systems, 
it is not known theoretically why the whole system of Alice and Bob is also described by quantum theory.
To answer such a question, 
some researchers study general probabilistic theories (GPTs), 
which are theoretical physical models defined on the basis of probability to obtain measurement outcomes, 
and contain quantum theory and classical probability theory \cite{Short,Janotta,Lami,Yoshida1,Muller,Masanes}.

Let $\cH(d_1,\ldots,d_n)$ be an $n$-partite complex Hilbert space $\mathbb{C}^{d_1}\otimes\cdots\otimes\mathbb{C}^{d_n}$, 
and $\tilde{d}$ be the dimension $d_1d_2\cdots d_n$ of $\cH(d_1,\ldots,d_n)$.
Each GPT has a unique number called \textit{capacity} (see Section~\ref{GPT}).
To derive the capacities of special GPTs, 
the following statement plays an important roll: 
\begin{itemize}
	\item[(S1)]
	for every unit vector $u\in\cH(d_1,\ldots,d_n)$, 
	the two matrices $I\pm\ketbra{u}{u}$ lie in the set 
	$\Sep(d_1,\ldots,d_n) \coloneqq \conv\{ X^{[1]}\otimes\cdots\otimes X^{[n]} : 
	X^{[1]},\ldots,X^{[n]}\text{ positive semi-definite} \}$,
\end{itemize}
where $I$ denotes the identity matrix on $\cH(d_1,\ldots,d_n)$ 
and $\conv(\cS)$ denotes the convex hull of a subset $\cS$.
Throughout this paper, 
we use the superscript $[j]$ (resp.\ $[k:l]$) to express the $j$th site (resp.\ the sites from $k$th to $l$th).
A matrix in $\Sep(d_1,\ldots,d_n)$ is called \textit{separable}.
Every separable matrix is positive semi-definite, 
but the converse does not necessarily hold.
If $n=2$, then (S1) is true for all integers $d_1,d_2\ge2$.
In fact, if $n=2$, then for all integers $d_1,d_2\ge2$, 
the following statement holds \cite{Gurvits}: 
\begin{itemize}
	\item[(S2)]
	the matrix $I+X$ lies in $\Sep(d_1,\ldots,d_n)$ 
	for every Hermitian matrix $X$ with $\|X\|_2 \coloneqq (\Tr X^\ast X)^{1/2}\le1$.
\end{itemize}
Statement (S2) is stronger than (S1) and does not hold in general \cite{Aubrun,Hildebrand}, 
but (S1) is still open for $n\ge3$ to the best of our knowledge.

Our main interest is whether (S1) holds for all integers $n\ge2$ and $d_1,\ldots,d_n\ge2$.
In this paper, we show the following statement weaker than (S1): 
for all integers $n\ge2$ and $d_1,\ldots,d_n\ge2$, 
\begin{itemize}
	\item[(S3)]
	every $(\tilde{d}-1)$-dimensional subspace $\cL$ of $\cH(d_1,\ldots,d_n)$ 
	has a product basis,
\end{itemize}
where a vector $u\in\cH(d_1,\ldots,d_n)$ is called a \textit{product vector} 
if $u=u^{[1]}\otimes\cdots\otimes u^{[n]}$ 
for some $u^{[1]}\in\mathbb{C}^{d_1},\ldots,u^{[n]}\in\mathbb{C}^{d_n}$; 
a basis composed of product vectors is called a \textit{product basis}.
In Section~\ref{proof}, we give a $(\tilde{d}-2)$-dimensional subspace with no product basis.
These results yield the following theorem.

\begin{theorem}\label{main0}
	For all integers $n\ge2$ and $d_1,\ldots,d_n\ge2$, 
	\[
	\max\Bigl\{ \dim\cL : 
	\begin{array}{l}
		\cL\text{ is a subspace of $\cH(d_1,\ldots,d_n)$ and}\\
		\text{has no product basis}
	\end{array}
	\Bigr\}
	= \tilde{d}-2.
	\]
\end{theorem}

Theorem~\ref{main0} is still true even if the scalar field $\mathbb{C}$ is replaced with 
an arbitrary infinite field (Remark~\ref{infiniteF}).
The case of finite fields $\cF$ is also true 
if either (i) $n=2$ or 
(ii) $n\ge3$ and $\#\cF>\max\{d_i : i\not=n_1,n_2\}$ for some $n_1$ and $n_2$ (Theorem~\ref{main1}), 
where $\#\cF$ denotes the order of $\cF$.
We address the case of finite fields in Section~\ref{finiteF}.

Actually, existing studies often consider an \textit{orthogonal product basis} \cite{Horodecki,Bennett,DiVincenzo, Alon,Chen,Feng}, 
which is defined as an orthonormal basis composed of product vectors.
An orthogonal product basis of a subspace $\cL$ is called \textit{unextendible} 
if the orthogonal complement of $\cL$ contains no non-zero product vector.
Unextendible orthogonal product bases (UPBs; ``orthogonal'' is usually omitted) are used to construct bound entangled states \cite{Horodecki,Bennett,DiVincenzo}.
In particular, quantum information theory motivates us to find UPBs of the minimum possible number.
Alon and Lov\'asz \cite{Alon} proved that 
the minimum dimension of subspaces of $\cH(d_1,\ldots,d_n)$ with UPBs is equal to $d_1+\cdots+d_n-n+1$ 
unless either (i) $n=2$ and $2\in\{d_1,d_2\}$ or 
(ii) $d_1+\cdots+d_n-n+1$ is odd and at least one $d_i$ is even.
Moreover, the minimum dimension is strictly greater than $d_1+\cdots+d_n-n+1$ in cases (i) and (ii).
After their work, cases (i) and (ii) have been studied in more detail \cite{Feng,Chen}.

Finally, we state two statements similar to Theorem~\ref{main0}.
A subspace of $\cH(d_1,\ldots,d_n)$ containing no non-zero product vector 
is called \textit{completely entangled}.
Wallach \cite{Wallach} and Parthasarathy \cite{Parthasarathy} proved that 
for all integers $n\ge2$ and $d_1,\ldots,d_n\ge2$, 
\begin{equation}
\begin{split}
	&\quad \max\{ \dim\cL : \cL\text{ is a completely entangled subspace of }\cH(d_1,\ldots,d_n) \}\\
	&= \tilde{d} - (d_1+\cdots+d_n) + n-1.
\end{split}\label{eq5}
\end{equation}
Cubitt et al.\ \cite{Cubitt} proved that 
for all integers $d_1,d_2\ge2$ and $r\in[0,\min\{d_1,d_2\}-1]$, 
\begin{equation}
	\max\biggl\{ \dim\cL : 
	\begin{array}{l}
		\cL\text{ is a subspace of }\cH(d_1,d_2)\text{ satisfying that}\\
		\sr u\ge r+1\text{ for all non-zero }u\in\cL
	\end{array}
	\biggr\} \le (d_1-r)(d_2-r), \label{eq6}
\end{equation}
where $\sr u$ denotes the Schmidt rank of $u\in\cH(d_1,d_2)$, i.e., 
a unique number $k$ such that $u$ is expressed as $u=\sum_{i=1}^k u_i^{[1]}\otimes u_i^{[2]}$ 
with two orthogonal systems $(u_i^{[1]})_{i=1}^k$ and $(u_i^{[2]})_{i=1}^k$.
If $r=0,1$, then \eqref{eq6} has equality due to \eqref{eq5}.
Recently, Bag et al.\ \cite{Bag} constructed subspaces that achieve equality in \eqref{eq6} for all $d_1,d_2\ge4$ and $r=1,2,3$.

\section{More general proposition and proof}\label{proof}

Let us begin with notational conventions.
A vector $u\in\mathbb{C}^d$ is expressed as a column vector.
Also, we use the bra-ket notation: for $u\in\mathbb{C}^d$, 
$\ket{u}$ and $\bra{u}$ denote the column vector $u$ and its conjugate transpose, respectively.
Hence, $\braket{\cdot|\cdot}$ gives the standard Hermitian inner product on $\mathbb{C}^d$, 
and $\ketbra{u}{u}$ is a rank-one orthogonal projection for every unit vector $u\in\mathbb{C}^d$.
Let $(e_i^{[j]})_{i=1}^{d_j}$ be the standard basis of $\mathbb{C}^{d_j}$ for $j=1,\ldots,n$.
Denote by $\vspan(\cS)$ the linear span of a subset $\cS$, 
and by $\cL^\bot$ the orthogonal complement of a subspace $\cL$.
Although product vectors have been already defined in the case $n\ge2$, 
all vectors are regarded as product vectors in the case $n=1$.

Now, we prove the following proposition which is more general than (S3).

\begin{proposition}\label{main0'}
	Let $n\ge1$ and $d_1,\ldots,d_n\ge2$ be integers, 
	and $r$ be an integer in the interval $[0, \min\{d_1,\ldots,d_n\}]$.
	If the dimension of a subspace $\cL$ of $\cH(d_1,\ldots,d_n)$ is greater than or equal to $\tilde{d}-r$, 
	then $\cL$ has a $(\tilde{d}-r^n)$-tuple $(u_i)_{i=1}^{\tilde{d}-r^n}$ of linearly independent product vectors.
\end{proposition}

To prove Proposition~\ref{main0'}, we need two lemmas.
The first one is basic in algebra, and 
the second one is proved by using the first one.

\begin{lemma}\label{lem1}
	Let $\cF$ be an infinite field, $n\ge1$ be an integer, 
	and $f(x_1,\ldots,x_n)$ be a polynomial over $\cF$.
	Then the following conditions are equivalent: 
	\begin{enumerate}
		\item
		$f(\alpha_1,\ldots,\alpha_n)=0$ for all $\alpha_1,\ldots,\alpha_n\in\cF$;
		\item
		$f(x_1,\ldots,x_n)=0$ as a polynomial.
	\end{enumerate}
\end{lemma}
\begin{proof}
	See \cite[Theorem~2.19]{BasicAlgebra}.
\end{proof}

\begin{lemma}\label{lem2}
	Let $m,n\ge1$, $d_1,\ldots,d_n\ge2$ and $r\in[1,\tilde{d}]$ be integers, and 
	let $(u_{k,l})_{l=1}^{\tilde{d}-r}$, $k=1,\ldots,m$, 
	be $(\tilde{d}-r)$-tuples of linearly independent vectors in $\cH(d_1,\ldots,d_n)$.
	Then there exist product vectors $v_1,\ldots,v_r\in\cH(d_1,\ldots,d_n)$ such that 
	\[
	\det[u_{k,1},\ldots,u_{k,\tilde{d}-r},v_1,\ldots,v_r] \not= 0
	\]
	for all $k=1,\ldots,m$.
\end{lemma}
\begin{proof}
	Let $(e_i)_{i=1}^{\tilde{d}}$ be the standard basis of $\cH(d_1,\ldots,d_n)$, 
	and let $n'=\tilde{d}^r$.
	Define the $m$ polynomials $f_k(x_1,\ldots,x_{n'})$ over $\mathbb{C}$ as 
	\begin{gather*}
		f_k(x_1,\ldots,x_{n'}) = \det[u_{k,1},\ldots,u_{k,\tilde{d}-r},v_1,\ldots,v_r]
		\quad(k=1,\ldots,m),\\
		v_i = v_i^{[1]}\otimes\cdots\otimes v_i^{[n]}\quad(i=1,\ldots,r),
	\end{gather*}
	where the variables $x_1,\ldots,x_{n'}$ correspond to 
	the $n'$ entries of $v_i^{[j]}$, $1\le i\le r$, $1\le j\le n$.
	Since $(u_{k,l})_{l=1}^{\tilde{d}-r}$, $k=1,\ldots,m$, are tuples of linearly independent vectors, 
	we have $f_k(\alpha_1,\ldots,\alpha_{n'})\not=0$ for some $\alpha_1,\ldots,\alpha_{n'}\in\mathbb{C}$ 
	corresponding to $v_1,\ldots,v_r\in\{e_1,\ldots,e_{\tilde{d}}\}$ 
	(note that $e_1,\ldots,e_{\tilde{d}}$ are product vectors).
	Therefore, for every $k=1,\ldots,m$, 
	the polynomial $f_k(x_1,\ldots,x_{n'})$ is not zero as a polynomial.
	Since the polynomial ring $\mathbb{C}[x_1,\ldots,x_{n'}]$ is an integral domain, 
	the product $f(x_1,\ldots,x_{n'}) \coloneqq \prod_{k=1}^m f_k(x_1,\ldots,x_{n'})$ is not also zero as a polynomial.
	Thus, Lemma~\ref{lem1} implies that 
	$f(\beta_1,\ldots,\beta_{n'})\not=0$ for some $\beta_1,\ldots,\beta_{n'}\in\mathbb{C}$.
	Taking the vectors $v_i^{[j]}$ corresponding to $\beta_1,\ldots,\beta_{n'}\in\mathbb{C}$, 
	we obtain desired product vectors $v_1,\ldots,v_r$.
\end{proof}

\begin{proof}[Proof of Proposition~$\ref{main0'}$]
	Since the case $r=0$ is clear, 
	we assume the condition $r\ge1$ in this proof.
	We show the proposition by induction on $n\ge1$.
	First, the case $n=1$ is trivial.
	Let $n\ge2$ and assume that the proposition is true for $n-1$.
	Then we show that the proposition is also true for $n$.
	Let the dimension of a subspace $\cL$ be greater than or equal to $\tilde{d}-r$.
	For some $w_1,\ldots,w_r\in\cL^\bot$, 
	the subspace $\cL$ can be expressed as 
	\[
	\cL = \{ u\in\cH(d_1,\ldots,d_n) : \forall i=1,\ldots,r,\ \braket{w_i|u}=0 \}.
	\]
	Now, take a basis $(u_k^{[1]})_{k=1}^{d_1}$ of $\mathbb{C}^{d_1}$ and set $\tilde{d}'=\tilde{d}/d_1$.
	Since the dimension of the subspace 
	\[
	\cL_k^{[2:n]} \coloneqq \{ u^{[2:n]}\in\cH(d_2,\ldots,d_n) : \forall i=1,\ldots,r,\ \braket{w_i|u_k^{[1]}\otimes u^{[2:n]}}=0 \}
	\]
	is greater than or equal to $\tilde{d}'-r$ for every $k=1,\ldots,d_1$, 
	the induction hypothesis implies that 
	$\cL_k^{[2:n]}$ has a $(\tilde{d}'-r^{n-1})$-tuple $(u_{k,l}^{[2:n]})_{l=1}^{\tilde{d}'-r^{n-1}}$ of linearly independent product vectors.
	Also, due to Lemma~\ref{lem2}, 
	we can take an $r^{n-1}$-tuple $(v_s^{[2:n]})_{s=1}^{r^{n-1}}$ of product vectors with the following condition: 
	\begin{equation}
		\forall k=1,\ldots,d_1,\ 
		\det[u_{k,1}^{[2:n]},\ldots,u_{k,\tilde{d}'-r^{n-1}}^{[2:n]},v_1^{[2:n]},\ldots,v_{r^{n-1}}^{[2:n]}] \not= 0.
		\label{eq1}
	\end{equation}
	Moreover, for every $s=1,\ldots,r^{n-1}$, 
	take a $(d_1-r)$-tuple $(v_{s,t}^{[1]})_{t=1}^{d_1-r}$ of linearly independent vectors in the subspace 
	\[
	\cL_s^{[1]} \coloneqq \{ u^{[1]}\in\mathbb{C}^{d_1} : \forall i=1,\ldots,r,\ \braket{w_i|u^{[1]}\otimes v_s^{[2:n]}}=0 \}.
	\]
	Note that the $r^{n-1}(d_1-r)$ vectors $v_{s,t}^{[1]}\otimes v_s^{[2:n]}$ are linearly independent.
	\par
	Let us show that the $\tilde{d}-r^n$ product vectors of $\cL$ 
	\begin{equation}
		u_k^{[1]}\otimes u_{k,l}^{[2:n]},\quad v_{s,t}^{[1]}\otimes v_s^{[2:n]}
		\quad\Bigl(
		\begin{array}{cc}
			1\le k\le d_1,& 1\le l\le\tilde{d}'-r^{n-1},\\
			1\le s\le r^{n-1},& 1\le t\le d_1-r
		\end{array}
		\Bigr) \label{eq9}
	\end{equation}
	are linearly independent.
	Suppose that $\tilde{d}-r^n$ scalars $\alpha_{k,l}$ and $\beta_{s,t}$ satisfy 
	\begin{equation}
		\sum_{k,l} \alpha_{k,l}u_k^{[1]}\otimes u_{k,l}^{[2:n]}
		+ \sum_{s,t} \beta_{s,t}v_{s,t}^{[1]}\otimes v_s^{[2:n]} = 0.
		\label{eq2}
	\end{equation}
	Since $(u_k^{[1]})_{k=1}^{d_1}$ is a basis of $\mathbb{C}^{d_1}$, 
	for all $s$ and $t$, 
	there exist scalars $\gamma_{s,t,k}$ such that 
	$v_{s,t}^{[1]} = \sum_k \gamma_{s,t,k}u_k^{[1]}$.
	Thus, \eqref{eq2} can be rewritten as follows: 
	\begin{align*}
		0 &= \sum_{k,l} \alpha_{k,l}u_k^{[1]}\otimes u_{k,l}^{[2:n]}
		+ \sum_{s,t,k} \beta_{s,t}\gamma_{s,t,k}u_k^{[1]}\otimes v_s^{[2:n]}\\
		&= \sum_k u_k^{[1]}\otimes\Bigl( \sum_l \alpha_{k,l}u_{k,l}^{[2:n]} + \sum_{s,t} \beta_{s,t}\gamma_{s,t,k}v_s^{[2:n]} \Bigr).
	\end{align*}
	Since $(u_k^{[1]})_{k=1}^{d_1}$ is a basis of $\mathbb{C}^{d_1}$, 
	we have 
	\[
	\sum_l \alpha_{k,l}u_{k,l}^{[2:n]} + \sum_{s,t} \beta_{s,t}\gamma_{s,t,k}v_s^{[2:n]} = 0
	\]
	for every $k=1,\ldots,d_1$.
	This and \eqref{eq1} imply that $\alpha_{k,l}=0$ for all $k$ and $l$.
	Thus, \eqref{eq2} turns to 
	$\sum_{s,t} \beta_{s,t}v_{s,t}^{[1]}\otimes v_s^{[2:n]} = 0$.
	Since the $r^{n-1}(d_1-r)$ vectors $v_{s,t}^{[1]}\otimes v_s^{[2:n]}$ are linearly independent, 
	it follows that $\beta_{s,t}=0$ for all $s$ and $t$.
	Therefore, the vectors \eqref{eq9} are linearly independent, 
	and the proposition is also true for $n$.
\end{proof}

Next, we construct a $(\tilde{d}-2)$-dimensional subspace with no product basis 
on the basis of the case $n=2$.

\begin{proposition}\label{prop1}
	For all integers $n\ge2$ and $d_1,\ldots,d_n\ge2$, 
	there exists a $(\tilde{d}-2)$-dimensional subspace of $\cH(d_1,\ldots,d_n)$ with no product basis.
\end{proposition}
\begin{proof}
	First, assuming $n=2$, 
	we show that the $(\tilde{d}-2)$-dimensional subspace 
	\begin{equation}
		\cL^{[1:2]} \coloneqq \vspan\bigl( \{ e_1^{[1]}\otimes e_1^{[2]} + e_2^{[1]}\otimes e_2^{[2]} \}
		\cup\{ e_i^{[1]}\otimes e_j^{[2]} : (i,j)\not=(1,1),(2,1),(2,2) \} \bigr)
		\label{eqL12}
	\end{equation}
	has no product basis.
	Take an arbitrary product vector $u=u^{[1]}\otimes u^{[2]}\in\cL^{[1:2]}$ with the expressions 
	$u^{[k]} = \sum_{i=1}^{d_k} \alpha_i^{[k]}e_i^{[k]}$, $\alpha_i^{[k]}\in\mathbb{C}$, $k=1,2$.
	Then $u$ is expressed in two ways: 
	\begin{align*}
		u &= \sum_{i,j} \alpha_i^{[1]}\alpha_j^{[2]}e_i^{[1]}\otimes e_j^{[2]}\\
		&= \beta(e_1^{[1]}\otimes e_1^{[2]} + e_2^{[1]}\otimes e_2^{[2]})
		+ \sum_{(i,j)\not=(1,1),(2,1),(2,2)} \alpha_i^{[1]}\alpha_j^{[2]}e_i^{[1]}\otimes e_j^{[2]}
		\quad(\exists\beta\in\mathbb{C}),
	\end{align*}
	where the second equality follows from the basis \eqref{eqL12} of $\cL^{[1:2]}$.
	This yields that 
	\[
	\alpha_1^{[1]}\alpha_1^{[2]}e_1^{[1]}\otimes e_1^{[2]}
	+ \alpha_2^{[1]}\alpha_1^{[2]}e_2^{[1]}\otimes e_1^{[2]}
	+ \alpha_2^{[1]}\alpha_2^{[2]}e_2^{[1]}\otimes e_2^{[2]}
	= \beta(e_1^{[1]}\otimes e_1^{[2]} + e_2^{[1]}\otimes e_2^{[2]}).
	\]
	Thus, it turns out that $\alpha_2^{[1]}\alpha_1^{[2]}=0$ and 
	$\alpha_1^{[1]}\alpha_1^{[2]} = \alpha_2^{[1]}\alpha_2^{[2]}$, 
	whence $\braket{e_1^{[1]}\otimes e_1^{[2]}|u} = \alpha_1^{[1]}\alpha_1^{[2]} = 0$.
	That is, $u$ is orthogonal to $e_1^{[1]}\otimes e_1^{[2]}$.
	However, the vector $e_1^{[1]}\otimes e_1^{[2]} + e_2^{[1]}\otimes e_2^{[2]}\in\cL^{[1:2]}$ 
	is not orthogonal to $e_1^{[1]}\otimes e_1^{[2]}$, 
	which implies that $\cL^{[1:2]}$ has no product basis.
	\par
	Next, consider the case $n\ge3$.
	We show that the $(\tilde{d}-2)$-dimensional subspace 
	\begin{equation*}
		\cL = \cL^{[1:2]}\otimes\vspan(u_0^{[3:n]}) + \cH(d_1,d_2)\otimes\vspan(u_0^{[3:n]})^\bot
	\end{equation*}
	has no product basis, 
	where $u_0^{[3:n]} \coloneqq e_1^{[3]}\otimes\cdots\otimes e_1^{[n]}$.
	Take an arbitrary product vector $u\in\cL$.
	Then $u$ is expressed as $u = u^{[1:2]}\otimes u_0^{[3:n]} + v^{[1:2]}\otimes v^{[3:n]}$ 
	with suitable vectors $u^{[1:2]}\in\cL^{[1:2]}$, $v^{[1:2]}\in\cH(d_1,d_2)$ and $v^{[3:n]}\in\vspan(u_0^{[3:n]})^\bot$.
	Since $u$ is a product vector, 
	so is $(I^{[1:2]}\otimes\bra{u_0^{[3:n]}})u = u^{[1:2]}$, 
	where $I^{[1:2]}$ denotes the identity matrix on $\cH(d_1,d_2)$.
	As already proved,  
	the product vector $u^{[1:2]}\in\cL^{[1:2]}$ is orthogonal to $e_1^{[1]}\otimes e_1^{[2]}$.
	Thus, $u$ is orthogonal to $e_1^{[1]}\otimes e_1^{[2]}\otimes u_0^{[3:n]}$.
	However, the vector $(e_1^{[1]}\otimes e_1^{[2]} + e_2^{[1]}\otimes e_2^{[2]})\otimes u_0^{[3:n]}\in\cL$ 
	is not orthogonal to $e_1^{[1]}\otimes e_1^{[2]}\otimes u_0^{[3:n]}$, 
	which implies that $\cL$ has no product basis.
\end{proof}

\begin{proof}[Proof of Theorem~$\ref{main0}$]
	Proposition~\ref{main0'} with $r=0,1$ and Proposition~\ref{prop1} yield the theorem immediately.
\end{proof}

\begin{remark}\label{infiniteF}
	Let us consider the case when 
	the scalar field $\mathbb{C}$ and the Hermitian inner product $\braket{v|u}=\sum_i \overline{v(i)}u(i)$ 
	are replaced with an arbitrary field $\cF$ 
	and the non-degenerate bilinear form $\braket{v,u}=\sum_i v(i)u(i)$, respectively.
	In this case, the proof of Proposition~\ref{prop1} works well.
	Moreover, if $\cF$ is infinite, 
	then the proof of Proposition~\ref{main0'} also works well 
	because (i) $\dim\cL+\dim\cL^\bot=d$ and $(\cL^\bot)^\bot=\cL$ for every subspace $\cL$ of $\cF^d${} 
	and (ii) Lemma~\ref{lem1} holds.
	The fact (i) is also true for every finite field $\cF${}, 
	but (ii) is false for every finite field 
	even if the polynomial $f(x_1,\ldots,x_n)$ is homogeneous.
	Nevertheless, a modified version of Lemma~\ref{lem1} holds for every finite field (see Lemma~\ref{lem1'}).
\end{remark}

Finally, we verify that (S1) implies (S3), 
which follows from the following proposition immediately.

\begin{proposition}\label{S1S3}
	For a subspace $\cL$ of $\cH(d_1,\ldots,d_n)$, consider the following conditions: 
	\begin{enumerate}
		\item
		the orthogonal projection $P_\cL$ onto $\cL$ lies in $\Sep(d_1,\ldots,d_n)$;
		\item
		$\cL$ has a product basis.
	\end{enumerate}
	The one direction ``$(1)\Rightarrow(2)$'' holds for all integers $n\ge2$ and $d_1,\ldots,d_n\ge2$ and subspaces $\cL$, 
	but the converse does not necessarily hold.
\end{proposition}
\begin{proof}
	$(1)\Rightarrow(2)$.
	See \cite[Theorem~2]{Horodecki} 
	(which is only the case $n=2$, but the case $n\ge3$ are also proved in the same way).
	\par
	$(2)\not\Rightarrow(1)$.
	Let $n\ge2$ and $d_1,\ldots,d_n\ge2$ be integers.
	Choose $\cL$ as the subspace spanned by the two vectors 
	$e_1^{[1]}\otimes e_1^{[2]}\otimes u^{[3:n]}$ and 
	$(e_1^{[1]}+e_2^{[1]})\otimes(e_1^{[2]}+e_2^{[2]})\otimes u^{[3:n]}$, 
	where $u^{[3:n]}$ be an arbitrary unit product vector in $\cH(d_3,\cdots,d_n)$.
	Then $\cL$ has the product basis 
	composed of $e_1^{[1]}\otimes e_1^{[2]}\otimes u^{[3:n]}$ and $(e_1^{[1]}+e_2^{[1]})\otimes(e_1^{[2]}+e_2^{[2]})\otimes u^{[3:n]}$.
	Also, the orthogonal projection $P_\cL$ is equal to 
	\[
	P_\cL = (\ketbra{e_1^{[1]}}{e_1^{[1]}}\otimes\ketbra{e_1^{[2]}}{e_1^{[2]}} + \ketbra{u^{[1:2]}}{u^{[1:2]}})\otimes\ketbra{u^{[3:n]}}{u^{[3:n]}},
	\]
	where $u^{[1:2]}$ is the unit vector 
	$(e_1^{[1]}\otimes e_2^{[2]} + e_2^{[1]}\otimes e_1^{[2]} + e_2^{[1]}\otimes e_2^{[2]})/\sqrt{3}$.
	Since the matrix $\ketbra{e_1^{[1]}}{e_1^{[1]}}\otimes\ketbra{e_1^{[2]}}{e_1^{[2]}} + \ketbra{u^{[1:2]}}{u^{[1:2]}}$ is not separable 
	(you can use the positive partial transpose criterion), 
	the orthogonal projection $P_\cL$ is not also separable.
\end{proof}

\section{Case of finite fields}\label{finiteF}

As already stated in Remark~\ref{infiniteF}, 
Theorem~\ref{main0} holds for every infinite field.
In this section, we consider the case of finite fields.
Let $\cF$ be a finite field of order $q$, 
$\braket{\cdot,\cdot}$ be the non-degenerate bilinear form $\braket{v,u}=\sum_i v(i)u(i)$, 
$\tilde{d}$ be the dimension $d_1d_2\cdots d_n$ of $\cF^{d_1}\otimes\cdots\otimes\cF^{d_n}$, 
and $(e_i^{[j]})_{i=1}^{d_j}$ be the standard basis of $\cF^{d_j}$ for $j=1,\ldots,n$.
We denote by $\cL^\bot$ the orthogonal compliment of a subspace $\cL$ 
with respect to the non-degenerate bilinear form $\braket{\cdot,\cdot}$.

\begin{proposition}\label{main1''}
	Let $d_1,d_2\ge2$ be integers.
	Then every $(d_1d_2-1)$-dimensional subspace of $\cF^{d_1}\otimes\cF^{d_2}$ 
	has a product basis.
\end{proposition}
\begin{proof}
	Let $\cL$ be a $(d_1d_2-1)$-dimensional subspace of $\cF^{d_1}\otimes\cF^{d_2}$.
	Taking a non-zero $w\in\cL^\bot$, we have 
	$\cL = \{ u\in\cF^{d_1}\otimes\cF^{d_2} : \braket{w,u}=0 \}$.
	\par\setcounter{count}{0}
	\noindent\textbf{Step~\num.}
	Let us consider the case $w=w_r \coloneqq \sum_{i=1}^r e_i^{[1]}\otimes e_i^{[2]}$ 
	with $1\le r\le\min\{d_1,d_2\}$.
	Set $u_0^{[2]} = \sum_{i=1}^r e_i^{[2]}$.
	In this case, the $d_1d_2-1$ product vectors of $\cL$ 
	\begin{equation}
		e_i^{[1]}\otimes e_j^{[2]},\quad
		(e_k^{[1]} - e_{k+1}^{[1]})\otimes u_0^{[2]}\quad
		\bigl( (i,j)\not=(1,1),(2,2),\ldots,(r,r),\ 1\le k\le r-1 \bigr)
		\label{eq3}
	\end{equation}
	are linearly independent, which is proved as follows.
	First, it is easily checked that all the vectors \eqref{eq3} are orthogonal to $w_r$.
	Next, suppose that $d_1d_2-1$ scalars $\alpha_{i,j}$ and $\beta_k$ satisfy 
	\begin{equation}
		\sum_{(i,j)\not=(1,1),(2,2),\ldots,(r,r)} \alpha_{i,j}e_i^{[1]}\otimes e_j^{[2]}
		+ \sum_{k=1}^{r-1} \beta_k(e_k^{[1]} - e_{k+1}^{[1]})\otimes u_0^{[2]}
		= 0. \label{eq8}
	\end{equation}
	Taking the inner product of \eqref{eq8} and $e_l^{[1]}\otimes e_l^{[2]}$ for $l=1,\ldots,r$, 
	we obtain that $\beta_1=0$ and $\beta_l - \beta_{l-1}=0$ for all $l=2,\ldots,r$.
	Thus, all $\beta_k$ are zero.
	Since the vectors $e_i^{[1]}\otimes e_j^{[2]}$, $(i,j)\not=(1,1),(2,2),\ldots,(r,r)$, are linearly independent, 
	all $\alpha_{i,j}$ are also zero.
	Therefore, the vectors \eqref{eq3} are linearly independent, 
	and $\cL$ has a product basis.
	\par
	\noindent\textbf{Step~\num.}
	Let us reduce the case of general $w$ to Step~1.
	For a matrix $A=(\alpha_{i,j})\in\cF^{d_1\times d_2}$, 
	define the vector $\vvec(A)\in\cF^{d_1}\otimes\cF^{d_2}$ as 
	$\vvec(A)=\sum_{i,j} \alpha_{i,j}e_i^{[1]}\otimes e_j^{[2]}$.
	Then $\vvec(PAQ^\top)=(P\otimes Q)\vvec(A)$ 
	for all matrices $A\in\cF^{d_1\times d_2}$, $P\in\cF^{d_1\times d_1}$ and $Q\in\cF^{d_2\times d_2}$, 
	where $Q^\top$ denotes the transpose of $Q$.
	Now, express $w$ as $w=\vvec(A)$ with $A\in\cF^{d_1\times d_2}$.
	For $r=0,1,\ldots,\min\{d_1,d_2\}$, 
	define the matrix $B_r=(\beta_{i,j})\in\cF^{d_1\times d_2}$ as 
	$\beta_{i,j}=1$ if $(i,j)=(1,1),(2,2),\ldots,(r,r)$ and $\beta_{i,j}=0$ otherwise.
	Since $A$ can be factorized as $A=P^\top B_rQ$ 
	with an integer $r\in[1,\min\{d_1,d_2\}]$ and invertible matrices $P\in\cF^{d_1\times d_1}$ and $Q\in\cF^{d_2\times d_2}${}, 
	it follows that 
	\[
	w = \vvec(A) = (P^\top\otimes Q^\top)\vvec(B_r) = (P^\top\otimes Q^\top)w_r.
	\]
	Letting $(u_i)_{i=1}^{d_1d_2-1}$ be the product basis \eqref{eq3}, 
	we find that $((P^{-1}\otimes Q^{-1})u_i)_{i=1}^{d_1d_2-1}$ is a product basis of $\cL$.
\end{proof}

\begin{proposition}\label{main1'}
	Let $n\ge3$ and $d_1,\ldots,d_n\ge2$ be integers.
	If $q>\max\{d_i : i\not=n_1,n_2\}$ for some $n_1$ and $n_2$, 
	then every $(\tilde{d}-1)$-dimensional subspace of $\cF^{d_1}\otimes\cdots\otimes\cF^{d_n}$ 
	has a product basis.
\end{proposition}

Since Proposition~\ref{prop1} holds for every finite field 
(see Remark~\ref{infiniteF}), 
we obtain the following theorem.

\begin{theorem}\label{main1}
	Let $n\ge2$ and $d_1,\ldots,d_n\ge2$ be integers.
	If either (i) $n=2$ or 
	(ii) $n\ge3$ and $q>\max\{d_i : i\not=n_1,n_2\}$ for some $n_1$ and $n_2$, 
	then 
	\[
	\max\Bigl\{ \dim\cL : 
	\begin{array}{l}
		\cL\text{ is a subspace of $\cF^{d_1}\otimes\cdots\otimes\cF^{d_n}$ and}\\
		\text{has no product basis}
	\end{array}
	\Bigr\}
	= \tilde{d}-2.
	\]
\end{theorem}

To prove Proposition~\ref{main1'}, 
we use the following lemmas instead of Lemmas~\ref{lem1} and \ref{lem2}.

\begin{lemma}\label{lem1'}
	Let $n\ge1$ be an integer, $f(x_1,\ldots,x_n)$ be a polynomial over $\cF$, 
	and $d_i$ be the degree of $f(x_1,\ldots,x_n)$ in $x_i$.
	If $q>\max\{d_1,\ldots,d_n\}$, 
	then the following conditions are equivalent: 
	\begin{enumerate}
		\item
		$f(\alpha_1,\ldots,\alpha_n)=0$ for all $\alpha_1,\ldots,\alpha_n\in\cF$;
		\item
		$f(x_1,\ldots,x_n)=0$ as a polynomial.
	\end{enumerate}
\end{lemma}
\begin{proof}
	See \cite[Theorem~2.19]{BasicAlgebra} 
	(where only Lemma~\ref{lem1} is proved but the proof works well for Lemma~\ref{lem1'}).
\end{proof}

\begin{lemma}\label{lem2'}
	Let $m\in[1,q-1]$, $n\ge1$ and $d_1,\ldots,d_n\ge2$ be integers, and 
	let $(u_{k,l})_{l=1}^{\tilde{d}-1}$, $k=1,\ldots,m$, 
	be $(\tilde{d}-1)$-tuples of linearly independent vectors in $\cF^{d_1}\otimes\cdots\otimes\cF^{d_n}$.
	Then there exists a product vector $v\in\cF^{d_1}\otimes\cdots\otimes\cF^{d_n}$ such that 
	\[
	\det[u_{k,1},\ldots,u_{k,\tilde{d}-1},v] \not= 0
	\]
	for all $k=1,\ldots,m$.
\end{lemma}
\begin{proof}
	This proof is almost the same as the proof of Lemma~\ref{lem2}. 
	Define the $m$ polynomials $f_k(x_1,\ldots,x_{\tilde{d}})$ over $\cF$ as 
	\begin{gather*}
		f_k(x_1,\ldots,x_{\tilde{d}}) = \det[u_{k,1},\ldots,u_{k,\tilde{d}-1},v]
		\quad(k=1,\ldots,m),\\
		v = v^{[1]}\otimes\cdots\otimes v^{[n]},
	\end{gather*}
	where the variables $x_1,\ldots,x_{\tilde{d}}$ correspond to 
	the $\tilde{d}$ entries of $v^{[j]}$, $1\le j\le n$.
	Then, for every $k=1,\ldots,m$, 
	the polynomial $f_k(x_1,\ldots,x_{\tilde{d}})$ is not zero as a polynomial.
	Since the polynomial ring $\cF[x_1,\ldots,x_{\tilde{d}}]$ is an integral domain, 
	the product $f(x_1,\ldots,x_{\tilde{d}}) \coloneqq \prod_{k=1}^m f_k(x_1,\ldots,x_{\tilde{d}})$ is not also zero as a polynomial.
	Also, for every $i=1,\ldots,\tilde{d}$, 
	the degree of the $f(x_1,\ldots,x_{\tilde{d}})$ in $x_i$ is less than or equal to $m\le q-1$.
	Thus, Lemma~\ref{lem1'} implies that 
	$f(\beta_1,\ldots,\beta_{\tilde{d}})\not=0$ for some $\beta_1,\ldots,\beta_{\tilde{d}}\in\cF$.
	Taking the vectors $v^{[j]}$ corresponding to $\beta_1,\ldots,\beta_{\tilde{d}}\in\cF$, 
	we obtain a desired product vector $v$.
\end{proof}

\begin{proof}[Proof of Proposition~$\ref{main1'}$]
	Without loss of generality, we may assume that $2\le d_1\le\cdots\le d_n$ and $q>d_{n-2}$ ($\{n_1,n_2\}=\{n-1,n\}$ in this case).
	The proposition is proved in the same way as the proof of Proposition~\ref{main0'} 
	by using Lemmas~\ref{lem1'} and \ref{lem2'} instead of Lemmas~\ref{lem1} and \ref{lem2}.
\end{proof}

\section{Capacities of GPTs}\label{GPT}

In this section, we describe the framework of GPTs, define the capacities of GPTs, 
and prove a few statements on capacities briefly.
The content of this section except for Proposition~\ref{capa2} 
relies on \cite{Yoshida1} (framework of GPTs) and \cite{QIT} (capacities) essentially.
M\"uller et al.\ \cite{Muller} and Masanes and M\"uller \cite{Masanes} also discussed the capacities of GPTs 
(in settings different from \cite{QIT}), 
but we think that \cite{QIT} is easier to understand for mathematicians because of fewer assumptions.

A GPT is given by (i) a real Hilbert space $\cV$ equipped with an inner product $\braket{\cdot,\cdot}$, 
(ii) a proper cone $\cK$ of $\cV$, and 
(iii) a vector $u$ in the interior of the dual cone $\cK^\ast$, 
where a subset $\cK$ of $\cV$ is called a \textit{convex cone} if $\alpha u+\beta v\in\cK$ for all $u,v\in\cK$ and $\alpha,\beta\ge0$;  
a convex cone $\cK$ is called \textit{proper} if $\cK$ is closed, has an interior point, and satisfies $\cK\cap(-\cK)=\{0\}$; 
for a convex cone $\cK$, the \textit{dual cone} $\cK^\ast$ is defined as 
\[
\cK^\ast = \{ y\in\cV : \forall x\in\cK,\ \braket{y,x}\ge0 \}.
\]
We only consider finite dimensions $\cV$.
It is known that 
\begin{itemize}
	\item
	if $\cK$ is a non-empty closed convex cone, then $\cK^{\ast\ast}=\cK$;
	\item
	if $\cK$ is a proper cone, then so is $\cK^\ast$.
\end{itemize}
The vector $u$ is called a \textit{unit effect} and fixed for each GPT.
Given $n$ GPTs $(\cV^{[i]},\cK^{[i]},u^{[i]})$ describing subsystems like Alice and Bob's systems, 
a GPT $(\cV,\cK,u)$ describing the whole system must satisfy that 
(i) $\cV=\cV^{[1]}\otimes\cdots\otimes\cV^{[n]}$, 
(ii) $\cK_{\min} \subset \cK \subset \cK_{\max}$, and 
(iii) $u=u^{[1]}\otimes\cdots\otimes u^{[n]}$, 
where $\cK_{\min}$ and $\cK_{\max}$ are defined as 
\begin{align*}
	\cK_{\min} &= \conv\{ x^{[1]}\otimes\cdots\otimes x^{[n]} : 
	x^{[1]}\in\cK^{[1]},\ldots,x^{[n]}\in\cK^{[n]} \},\\
	\cK_{\max} &= \{ x\in\cV : 
	\forall y^{[1]}\in(\cK^{[1]})^\ast,\ldots,\forall y^{[n]}\in(\cK^{[n]})^\ast,\ 
	\braket{x, y^{[1]}\otimes\cdots\otimes y^{[n]}}\ge0 \}.
\end{align*}
Conditions (i)--(iii) are naturally derived by considering local operations and randomization.
It is important that $\cK$ cannot be determined uniquely 
because $\cK_{\min}$ and $\cK_{\max}$ are not equal to each other in general.

For a GPT $(\cV,\cK,u)$, define the \textit{state class} $\cS(\cK,u)$ and \textit{measurement class} $\cM(\cK^\ast,u)$ as 
\begin{equation}
\begin{split}
	\cS(\cK,u) &= \{ x\in\cK : \braket{x,u}=1 \},\\
	\cM(\cK^\ast,u) &= \biggl\{ (y_i)_{i=1}^m\text{ $m$-tuple of elements in }\cK^\ast : m\in\mathbb{N},\ \sum_{i=1}^m y_i=u \biggr\}.
\end{split}\label{eq7}
\end{equation}
In the above definition, we can use another proper cone $\cK'\subset\cK^\ast$ instead of $\cK^\ast$, 
but the condition $\cK'=\cK^\ast$ is imposed in usual.
An element in $\cS(\cK,u)$ (resp.\ $\cM(\cK^\ast,u)$) is called a \textit{state} (resp.\ \textit{measurement}).
Also, each $i$ in $(y_i)_{i=1}^m\in\cM(\cK^\ast,u)$ represents a measurement outcome.
When a state $x\in\cS(\cK,u)$ is measured by a measurement $(y_i)_{i=1}^m\in\cM(\cK^\ast,u)$, 
the probability to obtain each outcome $i=1,\ldots,m$ is given by $\braket{x,y_i}$.
Indeed, $(\braket{x,y_i})_{i=1}^m$ is a probability vector thanks to the definition \eqref{eq7}.

\begin{example}[Quantum system]
	Let $\Herm(d)$ be the set of all Hermitian matrices on $\mathbb{C}^d$, 
	$\PSD(d)$ be the set of all positive semi-definite matrices on $\mathbb{C}^d$, 
	and $I$ be the identity matrix on $\mathbb{C}^d$.
	When we equip $\Herm(d)$ with the Hilbert-Schmidt inner product $\braket{X,Y}=\Tr XY$, 
	the tuple $(\Herm(d),\PSD(d),I)$ is a GPT called \textit{$d$-level quantum system}.
	Note that $\PSD(d)$ is \textit{self-dual}, i.e., $\PSD(d)^\ast=\PSD(d)$.
	The state class $\cS(\PSD(d),I)$ is the set of all density matrices, 
	and the measurement class $\cM(\PSD(d),I)$ is the set of all positive-operator valued measures (POVMs).
\end{example}

\begin{example}[Locally quantum system]
	Given $d_i$-level quantum subsystems, $i=1,\ldots,n$, 
	a GPT $(\cV,\cK,u)$ describing the whole system is called a \textit{$(d_1,\ldots,d_n)$-level locally quantum system} 
	if (i) $\cV$ and $u$ are equal to $\Herm(\tilde{d})$ and the identity matrix on $\cH(d_1,\ldots,d_n)$ respectively, 
	and (ii) $\Sep(d_1,\ldots,d_n)\subset\cK\subset\Sep(d_1,\ldots,d_n)^\ast$.
	In this case, $\cK_{\min}=\Sep(d_1,\ldots,d_n)$ and $\cK_{\max}=\Sep(d_1,\ldots,d_n)^\ast$.
	For all integers $n\ge2$ and $d_1,\ldots,d_n\ge2$, 
	the two proper cones $\Sep(d_1,\ldots,d_n)$ and $\Sep(d_1,\ldots,d_n)^\ast$ are not equal to each other.
	Since the inclusion relation $\Sep(d_1,\ldots,d_n)\subset\PSD(\tilde{d})\subset\Sep(d_1,\ldots,d_n)^\ast$ holds, 
	the $\tilde{d}$-level quantum system is a $(d_1,\ldots,d_n)$-level locally quantum system.
\end{example}

We next define the capacity of a GPT.

\begin{definition}[Simultaneously distinguishable states]
	Let $(\cV,\cK,u)$ be a GPT.
	We say that $m$ states $x_1,\ldots,x_m\in\cS(\cK,u)$ are simultaneously distinguishable 
	if there exists a measurement $(y_j)_{j=1}^m\in\cM(\cK^\ast,u)$ such that 
	$\braket{x_i,y_j}=\delta_{i,j}$ for all $i,j=1,\ldots,m$, 
	where $\delta_{i,j}$ denotes the Kronecker delta.
\end{definition}

\begin{definition}[Capacity]
	For a GPT $(\cV,\cK,u)$, 
	the maximum number of simultaneously distinguishable states is called the capacity.
	We denote by $\capa(\cV,\cK,u)$ the capacity of a GPT $(\cV,\cK,u)$.
\end{definition}

For example, it is known that the capacity of $d$-level quantum system is equal to $d$.
As proved below, the capacity of each $(d_1,d_2)$-level locally quantum system is equal to $d_1d_2$.
This fact is found in \cite{QIT} (without proof).

\begin{proposition}\label{capa1}
	Let $d_1,d_2\ge2$ be integers.
	For every $(d_1,d_2)$-level locally quantum system, the capacity is equal to $d_1d_2$.
\end{proposition}

Proposition~\ref{capa1} asserts that 
the capacities of locally quantum systems do not change in the bipartite case.
Another property on distinguishable states has been studied in \cite{Arai,Yoshida2}, 
which changes depending on locally quantum systems.

\begin{proof}[Proof of Proposition~$\ref{capa1}$]
	We use the fact that (S2) holds in the case $n=2$ \cite{Gurvits}.
	First, let us show that every $Y\in\Sep(d_1,d_2)^\ast$ satisfies $\|Y\|_2\le\Tr Y$.
	Since the case $Y=0$ is trivial, assume that $Y\in\Sep(d_1,d_2)^\ast$ is non-zero.
	Set $X=-Y/\|Y\|_2$. Then $I+X$ lies in $\Sep(d_1,d_2)$.
	Thus, $\Tr Y - \|Y\|_2 = \Tr(I+X)Y \ge 0$.
	\par
	Let $(\Herm(d_1d_2),\cK,I)$ be a $(d_1,d_2)$-level locally quantum system.
	Next, we show that the capacity is equal to $d_1d_2$.
	Suppose that $m$ states $\rho_1,\ldots,\rho_m\in\cS(\cK,I)$ are simultaneously distinguishable 
	by a measurement $(M_j)_{j=1}^m$.
	Then 
	\begin{equation}
	\begin{split}
		m &\overset{(a)}{=} \sum_{i=1}^m \Tr\rho_iM_i \le \sum_{i=1}^m \|\rho_i\|_2\,\|M_i\|_2
		\overset{(b)}{\le} \sum_{i=1}^m (\Tr\rho_i)(\Tr M_i)\\
		&\overset{(c)}{=} \sum_{i=1}^m \Tr M_i \overset{(d)}{=} \Tr I = d_1d_2,
	\end{split}\label{eq4}
	\end{equation}
	where $(a)$, $(b)$, $(c)$ and $(d)$ follow from the facts $\Tr\rho_iM_i=1$, $\rho_i,M_i\in\Sep(d_1,d_2)^\ast$, 
	$\Tr\rho_i=\Tr\rho_iI=1$ and $\sum_{i=1}^m M_i=I$, respectively.
	Therefore, the capacity is less than or equal to $d_1d_2$.
	Since the $d_1d_2$ states 
	\[
	\ketbra{e_i^{[1]}}{e_i^{[1]}}\otimes\ketbra{e_j^{[2]}}{e_j^{[2]}}\in\cS(\cK,I)
	\quad(1\le i\le d_1,\ 1\le j\le d_2)
	\]
	are simultaneously distinguishable 
	by the measurement $(\ketbra{e_i^{[1]}}{e_i^{[1]}}\otimes\ketbra{e_j^{[2]}}{e_j^{[2]}})_{i,j}\in\cM(\cK^\ast,I)$, 
	we find that the capacity is equal to $d_1d_2$.
\end{proof}

We have used (S2) with $n=2$ in the above proof, 
but (S2) is false in general \cite{Aubrun,Hildebrand}.
Instead of (S2), let us focus on (S1).
As already stated in Section~\ref{intro}, 
(S1) is still open for $n\ge3$ to the best of our knowledge.
Finally, assuming (S1), we derive the capacities of special locally quantum systems.

\begin{proposition}\label{capa2}
	Assume that (S1) is true for all integers $n\ge2$ and $d_1,\ldots,d_n\ge2$.
	Let $(\Herm(\tilde{d}),\cK,I)$ be a $(d_1,\ldots,d_n)$-level locally quantum system 
	satisfying either $\cK\subset\PSD(\tilde{d})$ or $\cK\supset\PSD(\tilde{d})$.
	Then $\capa(\Herm(\tilde{d}),\cK,I)$ is equal to $\tilde{d}$.
\end{proposition}
\begin{proof}
	First, let us show that every $Y\in\Sep(d_1,\ldots,d_n)^\ast$ satisfies $\|Y\|\le\Tr Y$, 
	where $\|\cdot\|$ denotes the operator norm.
	Due to (S1), it follows that 
	$\Tr Y \pm \braket{u |Y| u} = \Tr(I\pm\ketbra{u}{u})Y \ge 0$ 
	for every unit vector $u\in\cH(d_1,\ldots,d_n)$.
	Thus, $\|Y\| \le \Tr Y$.
	\par
	The remainder is almost the same as the proof of Proposition~\ref{capa1}.
	The difference between the proof of Proposition~\ref{capa1} and this proof is only \eqref{eq4}.
	We must change \eqref{eq4} as follows: 
	\begin{align*}
		m &= \sum_{i=1}^m \Tr\rho_iM_i \le
		\begin{cases}
			\sum_{i=1}^m \|\rho_i\|_1\,\|M_i\| & \cK\subset\PSD(\tilde{d}),\\
			\sum_{i=1}^m \|\rho_i\|\,\|M_i\|_1 & \cK\supset\PSD(\tilde{d})
		\end{cases}
		\\
		&\le \sum_{i=1}^m (\Tr\rho_i)(\Tr M_i)
		= \sum_{i=1}^m \Tr M_i = \Tr I = \tilde{d},
	\end{align*}
	where $\|\cdot\|_1$ denotes the trace norm.
\end{proof}

\section*{Acknowledgment}
The author was supported by JSPS KAKENHI Grant Number JP19J20161.


\begin{thebibliography}{10}

\bibitem{Short}
A.~J. Short and S.~Wehner.
\newblock Entropy in general physical theories.
\newblock {\em New J. Phys.}, 12(March):033023, 34, 2010.

\bibitem{Janotta}
P.~Janotta and H.~Hinrichsen.
\newblock Generalized probability theories: what determines the structure of
  quantum theory?
\newblock {\em J. Phys. A}, 47(32):323001, 32, 2014.

\bibitem{Lami}
L.~Lami, C.~Palazuelos, and A.~Winter.
\newblock Ultimate data hiding in quantum mechanics and beyond.
\newblock {\em Comm. Math. Phys.}, 361(2):661--708, 2018.

\bibitem{Yoshida1}
Y.~Yoshida and M.~Hayashi.
\newblock Asymptotic properties for {M}arkovian dynamics in quantum theory and
  general probabilistic theories.
\newblock {\em J. Phys. A}, 53(21):215303, 43, 2020.

\bibitem{Muller}
M.~P. M\"{u}ller, O.~C.~O. Dahlsten, and V.~Vedral.
\newblock Unifying typical entanglement and coin tossing: on randomization in
  probabilistic theories.
\newblock {\em Comm. Math. Phys.}, 316(2):441--487, 2012.

\bibitem{Masanes}
L.~Masanes and M.~P. M{\"u}ller.
\newblock A derivation of quantum theory from physical requirements.
\newblock {\em New J. Phys.}, 13(6):063001, 29, 2011.

\bibitem{Gurvits}
L.~Gurvits and H.~Barnum.
\newblock Largest separable balls around the maximally mixed bipartite quantum
  state.
\newblock {\em Phys. Rev. A}, 66(6):062311, 7, 2002.

\bibitem{Aubrun}
G.~Aubrun and S.~J. Szarek.
\newblock Tensor products of convex sets and the volume of separable states on
  $n$ qudits.
\newblock {\em Phys. Rev. A}, 73(2):022109, 10, 2006.

\bibitem{Hildebrand}
R.~Hildebrand.
\newblock Entangled states close to the maximally mixed state.
\newblock {\em Phys. Rev. A}, 75(6):062330, 10, 2007.

\bibitem{Horodecki}
P.~Horodecki.
\newblock Separability criterion and inseparable mixed states with positive
  partial transposition.
\newblock {\em Phys. Lett. A}, 232(5):333--339, 1997.

\bibitem{Bennett}
C.~H. Bennett, D.~P. DiVincenzo, T.~Mor, P.~W. Shor, J.~A. Smolin, and B.~M.
  Terhal.
\newblock Unextendible product bases and bound entanglement.
\newblock {\em Phys. Rev. Lett.}, 82(26, part 1):5385--5388, 1999.

\bibitem{DiVincenzo}
D.~P. DiVincenzo, T.~Mor, P.~W. Shor, J.~A. Smolin, and B.~M. Terhal.
\newblock Unextendible product bases, uncompletable product bases and bound
  entanglement.
\newblock {\em Comm. Math. Phys.}, 238(3):379--410, 2003.

\bibitem{Alon}
N.~Alon and L.~Lov\'{a}sz.
\newblock Unextendible product bases.
\newblock {\em J. Combin. Theory Ser. A}, 95(1):169--179, 2001.

\bibitem{Chen}
J.~Chen and N.~Johnston.
\newblock The minimum size of unextendible product bases in the bipartite case
  (and some multipartite cases).
\newblock {\em Comm. Math. Phys.}, 333(1):351--365, 2015.

\bibitem{Feng}
K.~Feng.
\newblock Unextendible product bases and 1-factorization of complete graphs.
\newblock {\em Discrete Appl. Math.}, 154(6):942--949, 2006.

\bibitem{Wallach}
N.~R. Wallach.
\newblock An unentangled {G}leason's theorem.
\newblock In {\em Quantum computation and information ({W}ashington, {DC},
  2000)}, volume 305 of {\em Contemp. Math.}, pages 291--298. Amer. Math. Soc.,
  Providence, RI, 2002.

\bibitem{Parthasarathy}
K.~R. Parthasarathy.
\newblock On the maximal dimension of a completely entangled subspace for
  finite level quantum systems.
\newblock {\em Proc. Indian Acad. Sci. Math. Sci.}, 114(4):365--374, 2004.

\bibitem{Cubitt}
T.~Cubitt, A.~Montanaro, and A.~Winter.
\newblock On the dimension of subspaces with bounded {S}chmidt rank.
\newblock {\em J. Math. Phys.}, 49(2):022107, 6, 2008.

\bibitem{Bag}
P.~Bag, S.~Dey, M.~Nagisa, and H.~Osaka.
\newblock The order-{$n$} minors of certain {$(n + k) \times n$} matrices.
\newblock {\em Linear Algebra Appl.}, 603:368--389, 2020.

\bibitem{BasicAlgebra}
N.~Jacobson.
\newblock {\em Basic algebra. {I}}.
\newblock W. H. Freeman and Company, New York, second edition, 1985.

\bibitem{QIT}
K.~Matsumoto and G.~Kimura.
\newblock On additivity of strong converse bound of noiseless channels in
  locally quantum systems ---in relation to the radius of the separable
  ball---.
\newblock In {\em Proc. of The 37th Quantum Information Technology Symposium
  (QIT37)}, pages 13--16, 2017.
\newblock \url{https://www.ieice.org/ken/paper/20171116Z1AP/eng/}.

\bibitem{Arai}
H.~Arai, Y.~Yoshida, and M.~Hayashi.
\newblock Perfect discrimination of non-orthogonal separable pure states on
  bipartite system in general probabilistic theory.
\newblock {\em J. Phys. A}, 52(46):465304, 14, 2019.

\bibitem{Yoshida2}
Y.~Yoshida, H.~Arai, and M.~Hayashi.
\newblock Perfect discrimination in approximate quantum theory of general
  probabilistic theories.
\newblock {\em Phys. Rev. Lett.}, 125(15):150402, 5 pp.--150406, 2020.

\end{thebibliography}
\end{document}